\documentclass[11pt]{amsart}
\usepackage{amscd}
\usepackage{amsmath, amssymb}
\usepackage{amsfonts}

\newcommand{\de}{\partial}

\newcommand{\Ric}{\mathrm{Ric}}

\newcommand{\db}{\overline{\partial}}
\newcommand{\ddbar}{\sqrt{-1}\de\db}

\newcommand{\R}{\mathbb{R}}

\newcommand{\tr}[2]{\textrm{tr}_{#1} #2}

\newcommand{\ti}[1]{\tilde{#1}}
\newcommand{\ve}{\varepsilon}
\newcommand{\vp}{\varphi}
\newcommand{\dist}{\mathrm{dist}}
\newcommand{\diam}{\mathrm{diam}}

\renewcommand{\leq}{\leqslant}
\renewcommand{\geq}{\geqslant}

\begin{document}
\newcounter{remark}
\newcounter{theor}
\setcounter{remark}{0}
\setcounter{theor}{1}
\newtheorem{claim}{Claim}
\newtheorem{thm}{Theorem}[section]
\newtheorem{prop}{Proposition}[section]
\newtheorem{lemma}{Lemma}[section]
\newtheorem{defn}{Definition}[theor]
\newtheorem{cor}{Corollary}[section]
\newenvironment{remark}[1][Remark]{\addtocounter{remark}{1} \begin{trivlist}
\item[\hskip
\labelsep {\bfseries #1  \thesection.\theremark}]}{\end{trivlist}}
\setlength{\arraycolsep}{2pt}

\title{Blowup behavior of the K\"ahler-Ricci flow on Fano manifolds}

\author{Valentino Tosatti}
\thanks{Supported in part by a Sloan Research Fellowship and NSF grant DMS-1236969.}
\address{Department of Mathematics, Northwestern University, 2033 Sheridan Road, Evanston, IL 60201}
\email{tosatti@math.northwestern.edu}
\begin{abstract}
	We study the blowup behavior at infinity of the normalized K\"ahler-Ricci flow on a Fano manifold which does not admit K\"ahler-Einstein metrics. We prove an estimate for the K\"ahler potential away from a multiplier ideal subscheme, which implies that the volume forms along the flow converge to zero locally uniformly away from the same set. Similar results are also proved for Aubin's continuity method.
\end{abstract}

\maketitle
\section{introduction}
Let $X$ be a Fano manifold of complex dimension $n$, which is a compact complex manifold with positive first Chern class $c_1(X)$, and let $\omega_0$ be a K\"ahler metric on $X$ with $[\omega_0]=c_1(X)>0$. Consider the normalized K\"ahler-Ricci flow, which is a flow of K\"ahler metrics $\omega_t$ in $c_1(X)$ which evolve by
\begin{equation}\label{flow1}
\frac{\de\omega_t}{\de t}=-\Ric(\omega_t)+\omega_t
\end{equation}
with initial condition $\omega_0$. Its fixed points are K\"ahler-Einstein (KE) metrics $\omega_{\mathrm{KE}}$
which satisfy $\Ric(\omega_{\mathrm{KE}})=\omega_{\mathrm{KE}}$, and it is known \cite{CS, PSS, ST, TZ} that
if $X$ admits a KE metric then the flow \eqref{flow1} converges smoothly to a (possibly different) KE metric.
On the other hand not every Fano manifold admits a KE metric, and a celebrated conjecture of Yau \cite{Y2}
predicts that this happens precisely when $(X,K_X^{-1})$ is stable in a suitable algebro-geometric sense.
The precise notion of stability is K-stability, introduced by Tian \cite{Ti2} and refined by Donaldson \cite{Do}.
Solutions of this conjecture by Chen-Donaldson-Sun \cite{CDS} and Tian \cite{Ti4} have appeared very recently.

We will also consider a different family of K\"ahler metrics $\ti{\omega}_t$ in $c_1(X)$
which solve Aubin's continuity method \cite{Au}
\begin{equation}\label{cont1}
\Ric(\ti{\omega}_t)=t\ti{\omega}_t+(1-t)\omega_0,
\end{equation}
with $t$ ranging in an interval inside $[0,1]$. We have that $\ti{\omega}_0=\omega_0$, and if \eqref{cont1} is
solvable up to $t=1$, then $\ti{\omega}_1$ is KE. On the other hand if no KE exists then \eqref{cont1} has a solution defined on a maximal interval $[0,R(X))$ where $R(X)\leq 1$ is an invariant of $X$ (independent of $\omega_0$)
characterized by Sz\'ekelyhidi \cite{Sz} as the greatest lower bound for the Ricci curvature of metrics in $c_1(X)$.

In this note we consider a Fano manifold $X$ which does not admit a KE metric, and investigate the question of the
behavior in this case of the K\"ahler-Ricci flow \eqref{flow1} as $t\to \infty$ or of the continuity method \eqref{cont1} as $t\to R(X)$.
In several recent works this question has been studied by reparametrizing the evolving metrics by diffeomorphisms and studying the geometric
limiting space \cite{Li, PSSW, SZ, TW, To}. The key point of this note is that we do not modify the evolving metrics by diffeomorphisms, but instead we want to understand the way in which they degenerate as tensors on the fixed complex manifold $X$.

To state our main result, let us introduce some notation.
The K\"ahler-Ricci flow \eqref{flow1} is equivalent to a flow of K\"ahler potentials in the following way. We have that $\omega_t=\omega_0+\sqrt{-1}\de\db\varphi_t$ where the functions $\vp_t$ evolve by
\begin{equation}\label{flow2}
\frac{\de\varphi_t}{\de t}=\log \frac{\omega_t^n}{\omega_0^n}+\varphi_t-h_0, \quad \vp_0=c_0,
\end{equation}
where $c_0$ is a suitable constant (defined in \cite[(2.10)]{PSS}) and where $h_0$ is the Ricci potential of $\omega_0$ (i.e. it satisfies $\Ric(\omega_0)-\omega_0=\sqrt{-1}\de\db h_0$ and $\int_X(e^{h_0}-1)\omega_0^n=0$).
Let us rewrite \eqref{flow2} as the following complex Monge-Amp\`ere equation
\begin{equation}\label{flow3}
(\omega_0+\sqrt{-1}\de\db\varphi_t)^n= e^{h_0-\varphi_t+\dot{\varphi}_t}\omega_0^n,
\end{equation}
where here and henceforth, we'll write $\dot{\varphi}_t=\frac{\de\varphi_t}{\de t}$.
The flow \eqref{flow2} has a global solution $\varphi_t$ for all $t\geq 0$ \cite{Ca}, and since
$X$ does not admit KE metrics, we must have
$\sup_{X\times[0,\infty)}\varphi_t=\infty$ (see e.g. \cite{PSS}). From now on
we fix a sequence of times $t_i\to\infty$ such that
\begin{equation}\label{supt}
\sup_{X\times[0,t_i]}\varphi_t=\sup_X\varphi_{t_i}\to\infty.
\end{equation}
For simplicity we'll write $\varphi_i=\varphi_{t_i}$ and $\omega_i=\omega_{t_i}$.

On the other hand if we write $\ti{\omega}_t=\omega_0+\ddbar\ti{\vp}_t$, then the continuity method \eqref{cont1} is equivalent to the complex Monge-Amp\`ere equation
\begin{equation}\label{cont2}
(\omega_0+\ddbar\ti{\vp}_t)^n =  e^{h_0 -t\ti{\vp}_t}\omega_0^n.
\end{equation}
If $X$ does not admit KE metrics then a solution $\ti{\vp}_t$ exists for $t\in [0,R(X))$ with $0<R(X)\leq 1$, and
$\sup_X\ti{\vp}_t\to \infty$ as $t$ approaches $R(X)$. We then fix a sequence $t_i\in [0,R(X))$ with $t_i\to R(X)$
and write $\ti{\vp}_i=\ti{\vp}_{t_i}$ and $\ti{\omega}_i = \ti{\omega}_{t_i}$.

In \cite{Na2}, Nadel proved that there is a proper analytic subvariety $S\subset X$ (a suitable multiplier ideal subscheme \cite{Na1}) such that
the measures $\ti{\omega}^n_i$ converge (as measures) to zero on compact subsets of $X\backslash S$.
More recently, the same statement was proved for the measures $\omega^n_i$ along the K\"ahler-Ricci flow by Clarke-Rubinstein \cite[Lemma 6.5]{CR} (see also \cite{Pa} for
a weaker statement). It is natural
to ask whether this convergence can be improved. In this note, we show that away from a possibly larger proper analytic subvariety
the measures $\omega^n_i$ and $\ti{\omega}^n_i$ converge to zero uniformly on compact sets.
More precisely, we have:

\begin{thm}\label{main}
Assume that $X$ is a Fano manifold that does not admit a K\"ahler-Einstein metric, and let $\varphi_i, \omega_i$ be defined as above. Then for any $\ve>0$ there is a proper nonempty analytic subvariety $S_\ve\subset X$ and a subsequence of $\varphi_i$ (still denoted by $\varphi_i$) such that given any compact set $K\subset X\backslash S_\ve$ there is a constant $C$ that depends only on $K,\ve,\omega_0$ such that
for all $x\in K$ and for all $i$ we have
\begin{equation}\label{stima}
-\varphi_i(x)+(1-\ve)\sup_X \varphi_i\leq C.
\end{equation}
In particular, $\varphi_i$ goes to plus infinity locally uniformly outside $S_\ve$,
and the volume forms $\omega_i^n$ converge to zero in the same sense.
Finally, the same properties hold for $\ti{\vp}_i$ and $\ti{\omega}_i$ which solve Aubin's continuity method.
\end{thm}
Moreover we can identify the subvariety $S_\ve$ as follows: from weak compactness of currents, there exists $\psi$ an $L^1$ function on $X$ which is $\omega_0$-plurisubharmonic, such that a subsequence of $\varphi_i-\sup_M\varphi_i$ converges to $\psi$ in $L^1$. Then we have that
$$S_\ve=V\left(\mathcal{I}\left(\frac{C}{\ve}\psi\right)\right),$$
where $C$ is a constant that depends only on $\omega_0$, and $\mathcal{I}$ denotes the multiplier ideal sheaf.
We note here that in the results of \cite{Na2} and \cite{CR} the multiplier ideal sheaf that enters is $\mathcal{I}(\gamma\psi)$, with $\frac{n}{n+1}<\gamma<1$, which gives a smaller
subvariety.

Finally let us remark that we expect Theorem \ref{main} to hold also when $\ve=0$, but our arguments below can only prove this when $n=1$ (in which case the theorem is empty
because there is just one
Fano manifold, $\mathbb{CP}^1$, which does admit a KE metric).
In fact more should be true: Tian's conjectural ``partial $C^0$ estimate'' for the continuity
method \cite{Ti4}
roughly says that $-\ti{\varphi}_i+\sup_X \ti{\varphi}_i$ should blow up at most logarithmically as we approach a subvariety. The
partial $C^0$ estimate was proved by Tian \cite{Ti3} for K\"ahler-Einstein Fano surfaces and more recently by Chen-Wang \cite{CW}
for the K\"ahler-Ricci flow on Fano surfaces. Very recently it was proved by Donaldson-Sun \cite{DS} for K\"ahler-Einstein metrics on Fano manifolds, by Chen-Donaldson-Sun \cite{CDS} and Tian \cite{Ti4}
for conic K\"ahler-Einstein metrics and by Phong-Song-Sturm \cite{PSS2} for shrinking K\"ahler-Ricci solitons.
\\

\noindent
{\bf Acknowledgements. }Most of this work was carried out while the
author was visiting the Morningside Center
of Mathematics in Beijing in 2007, which he would like to thank for the hospitality. He is
also grateful to S.-T. Yau for many discussions, to D.H. Phong for support and encouragement,
and to B. Weinkove for useful comments.

\section{Proof of the main theorem}

Before we start the proof of the main theorem, we need to recall a few estimates which
are known to hold along the K\"ahler-Ricci flow on Fano manifolds. The first one is the bound
\begin{equation}\label{pere1}
|\dot{\varphi}_t|\leq C,
\end{equation}
which holds for all $t\geq 0$, and was proved by Perelman (see \cite{ST}). It uses crucially the choice of $c_0$ in \eqref{flow2} given by \cite[(2.10)]{PSS}.
We will also need the following uniform Sobolev inequality \cite{Ye, Zh}
\begin{equation}\label{pere3}
\left(\int_X |f|^{\frac{2n}{n-1}}\omega_t^n\right)^{\frac{n-1}{n}}
\leq C_S \left(\int_X |\nabla f|^2_{\omega_t}\omega_t^n +\int_X |f|^2\omega_t^n
\right),
\end{equation}
which holds for all $t\geq 0$ and for all $f\in C^{\infty}(X)$, for a constant $C_S$ that depends only on $\omega_0$.
The following Harnack inequality \cite{Ru} will also be used
\begin{equation}\label{harnack}
-\inf_X \varphi_t\leq C+n\sup_X \varphi_t,
\end{equation}
which again holds for all $t\geq 0$.
Finally, we will use the following basic result:
\begin{prop}[Tian \cite{Ti}] \label{alpha}
Let $(X,\omega)$ be a compact K\"ahler manifold.
For any fixed $\lambda>0$ and for any sequence
$\varphi_i$ of K\"ahler potentials for $\omega$, there exists a subsequence, still denoted by $\varphi_i$, and a proper subvariety
$S\subset X$ such that for any $p\in X\backslash S$ there exists $r,C>0$ that depend
only on $\lambda,\omega$ and $p$, such that
\begin{equation}\label{integralozzo}
\int_{B_\omega (p,r)}e^{-\lambda(\varphi_i-\sup_X\varphi_i)}\omega^n\leq C.
\end{equation}
Also, the constants $r,C$ are uniform when $p$ ranges in a compact set of $X\backslash S$.
\end{prop}

\begin{proof}[Proof of Theorem \ref{main}]
The starting point is the parabolic analogue of the Aubin-Yau's $C^2$ estimate (see e.g. \cite{Ca, Ya}), which says that there exists a constant $C$ that depend only on $\omega_0$ such that for all $t\geq 0$ we have
\begin{equation}\label{c2}
\tr{\omega_0}{\omega_t}\leq C e^{C\varphi_t-(C+1)\inf_{X\times[0,t]}\varphi_s}.
\end{equation}
Here and in the following we will denote by $C$ a uniform positive constant which might change from line to line.
Combining \eqref{c2} with \eqref{harnack} and \eqref{supt} we get
$$\tr{\omega_0}{\omega_i}\leq C e^{C\sup_X\varphi_i+C\sup_{X\times[0,t_i]}\varphi_s}=C e^{C\sup_X\varphi_i}.$$
Notice that for any two K\"ahler metrics $\eta,\chi$ we always have that
$$\tr{\eta}{\chi}\leq \frac{1}{(n-1)!}(\tr{\chi}{\eta})^{n-1}\frac{\chi^n}{\eta^n},$$
and so in our case
$$\tr{\omega_i}{\omega_0}\leq Ce^{C\sup_X\varphi_i}\frac{\omega_0^n}{\omega_i^n}.$$
Using the Monge-Amp\`ere equation \eqref{flow3} and the estimate \eqref{pere1} we get
\begin{equation}\label{boundn}
\tr{\omega_i}{\omega_0}\leq Ce^{C\sup_X\varphi_i+\varphi_i}\leq C_0e^{D\sup_X\varphi_i},
\end{equation}
where $C_0$ and $D$ are uniform constants. An alternative derivation of \eqref{boundn} can be
obtained by evolving the quantity $\log \tr{\omega_t}{\omega_0}-A\vp_t$, with $A$ large, to obtain
$$\tr{\omega_t}{\omega_0}\leq C e^{A(\varphi_t-\inf_{X\times[0,t]}\varphi_s)},$$
and using
again \eqref{harnack}. As an aside, note that when $n=1$ we can actually choose $D=1$.

Let $A>0$ be a constant, to be determined later, and compute
\begin{equation}\label{ellineq}
\Delta_{\omega_i}e^{-A\vp_i}\geq-Ae^{-A\varphi_i}\Delta_{\omega_i}\varphi_i\geq -nAe^{-A\varphi_i}.
\end{equation}

\begin{prop}\label{moser}
For any $x\in X$ and $0<r<\diam(X,\omega_i)/2$ there is a constant $C$ that
depends only on $A,\omega_0$ such that
\begin{equation}\label{stima6}
\sup_{B_{\omega_i}(x,r/2)}e^{-A\varphi_i}\leq \frac{C}{r^{2n}} \int_{B_{\omega_i}(x,r)}e^{-A\varphi_i}\omega_i^n,
\end{equation}
holds for all $i$.
\end{prop}
\begin{proof}
We apply the method of Moser iteration to the inequality \eqref{ellineq}. The method is standard, except for the fact that $r$ could be bigger than the injectivity radius of $\omega_i$  and so the balls
$B_{\omega_i}(x,r)$ need not be diffeomorphic to Euclidean balls, and in fact might not even be smooth domains. From now on let $i$ be fixed, and fix two positive numbers $\rho<R<\diam(X,\omega_i)/2$. Then let $\eta$ be a cutoff function of the form
$\eta(y)=\psi(\dist_{\omega_i}(y,x))$ where $\psi$ is a smooth nonincreasing function from $\R$ to $\R$ such that $\psi(y)=1$ for $y\leq \rho$, $\psi(y)=0$ for $y\geq (R+\rho)/2$ and
$$\sup_\R |\psi'|\leq \frac{4}{R-\rho}.$$
Then $\eta$ is Lipschitz, equal to $1$ on $B_{\omega_i}(x,\rho)$, supported
inside $B_{\omega_i}(x,R)$ and satisfies
$$|\nabla \eta|_{\omega_i}\leq \frac{4}{R-\rho}$$
almost everywhere. For simplicity of notation we will let $f=e^{-A\varphi_i}$
and suppress all references to the metric $\omega_i$. So we can write
\eqref{ellineq} as
\begin{equation}\label{ellineq2}
\Delta f\geq -nAf.
\end{equation}
Then for any $p\geq 2$ we compute
$$\int_{B(x,R)}\eta^2|\nabla (f^{p/2})|^2=\frac{p^2}{4(p-1)}\int_{B(x,R)}
\eta^2\langle \nabla(f^{p-1}),\nabla f\rangle.$$
At this point we want to integrate by parts, and we can do this because of the following argument: we can exhaust $B(x,R)$ with an increasing sequence of subdomains $B_j$, $j=1,2,\dots$, that have smooth boundary.
Then we can apply Stokes' Theorem to each $B_j$, and when $j$ is sufficiently
large $\eta$ will vanish on $\de B_j$ so we get
$$\int_{B_j}
\eta^2\langle \nabla(f^{p-1}),\nabla f\rangle=-2\int_{B_j}\eta f^{p-1}
\langle \nabla\eta,\nabla f\rangle-\int_{B_j}\eta^2 f^{p-1}\Delta f.$$
Then we can let $j$ go to infinity and by dominated convergence the integrals
on $B_j$ converge to the same integrals on $B(x,R)$. Thus
$$\int_{B(x,R)}\eta^2|\nabla (f^{p/2})|^2=-\frac{p^2}{4(p-1)}\int_{B(x,R)}
\eta f^{p-1}(2\langle \nabla\eta,\nabla f\rangle+\eta\Delta f).$$
Using \eqref{ellineq2} and the Cauchy-Schwarz and Young inequalities we have that
\begin{equation}\label{stima3}
\begin{split}
\int_{B(x,R)}\eta^2|\nabla (f^{p/2})|^2&\leq C p \int_{B(x,R)}\eta^2 f^p
+p\int_{B(x,R)}f^p|\nabla\eta|^2\\
&+\frac{p}{4}\int_{B(x,R)}\eta^2 f^{p-2}|\nabla f|^2,
\end{split}
\end{equation}
where $C$ depends only on $A,n$.

The last term in \eqref{stima3} is equal to $\frac{1}{p}\int_{B(x,R)}\eta^2|\nabla (f^{p/2})|^2$
and so can be absorbed in the left hand side. We thus get
\begin{equation}\label{stima1}
\begin{split}
\int_{B(x,R)}\eta^2|\nabla (f^{p/2})|^2&\leq C p \int_{B(x,R)}\eta^2 f^p+\frac{Cp}{(R-\rho)^2}\int_{B(x,R)}f^p\\
&\leq \frac{Cp}{(R-\rho)^2}\int_{B(x,R)}f^p,
\end{split}
\end{equation}
as long as $R-\rho$ is small.
Now we use the Cauchy-Schwarz and Young inequalities again to bound
\begin{equation}\label{stima2}
\begin{split}
\int_{B(x,R)}|\nabla(\eta f^{p/2})|^2&\leq 2\int_{B(x,R)}\eta^2|\nabla (f^{p/2})|^2+ 2\int_{B(x,R)}f^p|\nabla\eta|^2\\
&\leq 2\int_{B(x,R)}\eta^2|\nabla (f^{p/2})|^2+\frac{C}{(R-\rho)^2}\int_{B(x,R)}f^p.
\end{split}
\end{equation}
Combining \eqref{stima1} and \eqref{stima2} we have
$$\int_{B(x,R)}|\nabla(\eta f^{p/2})|^2\leq \frac{Cp}{(R-\rho)^2}\int_{B(x,R)}f^p.$$
This together with the Sobolev inequality \eqref{pere3} gives
\begin{equation}\label{iteration2}
\begin{split}
\left(\int_{B(x,R)} \eta^{2\beta}f^{p\beta}\right)^{1/\beta}&\leq C
\int_{B(x,R)}|\nabla(\eta f^{p/2})|^2+C\int_{B(x,R)}\eta^2 f^p\\
&\leq \frac{Cp}{(R-\rho)^2}\int_{B(x,R)}f^p,
\end{split}
\end{equation}
where we write $\beta=n/(n-1)>1$. Raising this to the $1/p$ gives
\begin{equation}\label{iteration}
\left(\int_{B(x,\rho)} f^{p\beta}\right)^{1/p\beta}\leq \frac{C^{1/p}p^{1/p}}{(R-\rho)^{2/p}}\left(\int_{B(x,R)} f^{p}\right)^{1/p}.
\end{equation}
For each $j\geq 0$ we now set $p_j=2\beta^j$ and $R_j=\rho+\frac{R-\rho}{2^j}$. Setting
$p=p_j$, $R=R_{j}$, $\rho=R_{j+1}$ in \eqref{iteration} and iterating
(notice that $(R_{j+1}-R_j)=(R-\rho)2^{-j-1}$ is small) we easily get
\begin{equation*}
\begin{split}
\sup_{B(x,\rho)}f&\leq \frac{C}{(R-\rho)^{n}}
\left(\int_{B(x,R)} f^{2}\right)^{1/2}\\
&\leq
 \frac{C}{(R-\rho)^{n}}\left(\sup_{B(x,R)}f\right)^{1/2}
\left(\int_{B(x,R)} f\right)^{1/2}.
\end{split}
\end{equation*}
Using Young's inequality we see that
$$\sup_{B(x,\rho)}f\leq \frac{1}{2}\sup_{B(x,R)}f+\frac{C}{(R-\rho)^{2n}}
\int_{B(x,R)} f.$$
From here a standard iteration argument (see e.g. \cite[Lemma 3.4]{HL})
implies that
$$\sup_{B(x,\rho)}f\leq\frac{C}{(R-\rho)^{2n}}
\int_{B(x,R)} f,$$
and finally setting $\rho=r/2$, $R=r$ we get \eqref{stima6}.
\end{proof}

We now apply Proposition \ref{alpha} with $\lambda=A+1$ and get
an analytic subvariety $S$ with the property that given any compact set $K\subset X\backslash S$
there exists $r_0>0$ such that for any $x\in K$ we have that $B_{\omega_0}(x,r_0)\Subset X\backslash S$
and
\begin{equation}\label{stima7}
\int_{B_{\omega_0} (x,r_0)}e^{-(A+1)(\varphi_i-\sup_X\varphi_i)}\omega_0^n\leq C
\end{equation}
holds for all $i$.
We now let $r_i=r_0(C_0e^{D\sup_X\vp_i})^{-1/2},$ so that \eqref{boundn} implies that
\begin{equation}\label{stima8}
B_{\omega_i}(x,r_i)\subset B_{\omega_0}(x,r_0),
\end{equation}
so in particular $r_i<\diam(X,\omega_i)/2$.
Therefore we can apply Proposition \ref{moser}, and combining \eqref{stima6}, \eqref{stima7} and \eqref{stima8} we obtain
\begin{equation*}
\begin{split}
e^{-A\varphi_i(x)}&\leq \sup_{B_{\omega_i}(x,r_i/2)}e^{-A\varphi_i}\\
&\leq \frac{C}{r_i^{2n}}\int_{B_{\omega_i}(x,r_i)}
e^{-A\vp_i}\omega_i^n\\
&\leq C e^{nD\sup_X\vp_i}\int_{B_{\omega_0}(x,r_0)} e^{-(A+1)\vp_i}\omega_0^n\\
&= C e^{(nD-A-1)\sup_X\varphi_i}
\int_{B_{\omega_0}(x,r_0)}
e^{-(A+1)(\varphi_i-\sup_X\varphi_i)}\omega_0^n\\
&\leq C e^{(nD-A-1)\sup_X\varphi_i},
\end{split}
\end{equation*}
where $C$ depends only on given data and $A$. Taking log gives
$$-A\varphi_i(x)+(A-nD+1)\sup_X\varphi_i\leq C.$$
We now let $A=\frac{nD}{\ve}$ and divide by $A$ (keeping in mind that $\sup_X\vp_i>0$), and obtain the desired bound
$$-\varphi_i(x)+(1-\ve)\sup_X\varphi_i\leq C,$$
where $C$ does not depend on $i$ or on $x\in K\subset X\backslash S$, and
where the subvariety $S=S_\ve$ now depends on $\ve$.
This, together with \eqref{pere1}, immediately implies that the volume form $\omega_i^n=\omega_0^ne^{h_0-\vp_i+\dot{\vp}_{t_i}}$ goes to zero locally uniformly on $X\backslash S$.
\end{proof}
We now identify the subvariety $S_\ve$: from its definition,
that is from Proposition \ref{alpha}, we see that $S_\ve$ is equal to the multiplier ideal subscheme of the sequence $\varphi_i$ with exponent $(nD/\ve +1)$ as defined by Nadel in \cite{Na1}.
But the main Theorem in \cite{DK} then shows that this is the same as the multiplier ideal subscheme defined by $(nD/\ve +1)\psi$ where $\psi$ is a weak limit of $\varphi_i-\sup_M\varphi_i$.

The same proof as above goes through with minimal changes in the case of Aubin's continuity method, that is for the functions $\ti{\vp}_i$ and the metrics
$\ti{\omega}_i$.
We just need to justify why estimates analogous to \eqref{pere3} and \eqref{harnack} hold.
To see these, note that the metrics $\ti{\omega}_i$ satisfy $\Ric(\ti{\omega}_i)\geq C^{-1}\ti{\omega}_i>0$, so the
Bonnet-Myers theorem gives us the estimate $\diam(X,\ti{\omega}_i)\leq C$, which together with the fact that
their volume is fixed allows us to apply a result of Croke \cite{Cr} which gives a uniform Sobolev inequality of
the form \eqref{pere3} for the metrics $\ti{\omega}_i$. The Harnack inequality \eqref{harnack} in this case is proved in \cite{Ti}.


\begin{thebibliography}{99}
\bibitem{Au} Aubin, T. {\em R\'eduction du cas positif de l'\'equation de Monge-Amp\`ere sur les vari\'et\'es k\"ahl\'eriennes compactes \`a la d\'emonstration d'une in\'egalit\'e}, J. Funct. Anal. {\bf 57} (1984), no. 2, 143--153.
\bibitem{Ca} Cao, H.-D. {\em Deformation of K\"ahler metrics to K\"ahler-Einstein metrics on compact K\"ahler manifolds}, Invent. Math. {\bf 81}  (1985), no. 2, 359--372.
\bibitem{CDS} Chen, X., Donaldson, S.K., Sun, S. {\em K\"ahler-Einstein metrics and stability}, Int. Math. Res. Not. IMRN {\bf 2013}, Art. ID rns279, 7 pp.
\bibitem{CW}  Chen, X., Wang, B. {\em The K\"ahler Ricci flow on Fano manifolds (I)}, J. Eur. Math. Soc. (JEMS) {\bf 14} (2012), no. 6, 2001--2038.
\bibitem{CR} Clarke, B., Rubinstein, Y. {\em  Ricci flow and the metric completion of the space of K\"ahler metrics}, to appear in Amer. J. Math., arXiv:1102.3787.
\bibitem{Cr} Croke, C.B. {\em Some isoperimetric inequalities and eigenvalue estimates}, Ann. Sci. \`Ecole Norm. Sup. (4) {\bf 13} (1980), no. 4, 419--435.
\bibitem{CS} Collins, T.C., Sz\'ekelyhidi, G. {\em The twisted K\"ahler-Ricci flow}, arXiv:1207.5441.
\bibitem{DK} Demailly, J.-P., Koll\'ar, J. {\em Semi-continuity of complex singularity exponents and K\"ahler-Einstein metrics on Fano orbifolds}, Ann. Sci. \'Ecole Norm. Sup. (4) {\bf 34} (2001), no. 4, 525--556.
\bibitem{Do} Donaldson, S.K. {\em Scalar curvature and stability of toric varieties}, J. Differential Geom. {\em 62} (2002), no. 2, 289--349.
\bibitem{DS} Donaldson, S.K., Sun, S. {\em Gromov-Hausdorff limits of K\"ahler manifolds and algebraic geometry}, arXiv:1206.2609.
\bibitem{HL} Han, Q., Lin, F. {\em Elliptic partial differential equations}, AMS 1997.
\bibitem{Li} Li, C. {\em On the limit behavior of metrics in continuity method to K\"ahler-Einstein problem in toric Fano case}, to appear in Compos. Math., arXiv:1012.5229.
\bibitem{Na1} Nadel, A.M. {\em Multiplier ideal sheaves and K\"ahler-Einstein metrics of positive scalar curvature}, Ann. of Math. (2) {\bf 132} (1990), no. 3, 549--596.
\bibitem{Na2} Nadel, A.M. {\em Multiplier ideal sheaves and Futaki's invariant}, in \emph{Geometric theory of singular phenomena in partial differential equations (Cortona, 1995)},  7--16, Cambridge Univ. Press, 1998.
\bibitem{Pa} Pali, N. {\em Characterization of Einstein-Fano manifolds via the K\"ahler-Ricci flow}, Indiana Univ. Math. J. {\bf 57} (2008), no. 7, 3241--3274.
\bibitem{PSS} Phong, D.H., \v Se\v sum, N., Sturm, J. {\em Multiplier ideal sheaves and the K\"ahler-Ricci flow}, Comm. Anal. Geom. {\bf 15} (2007), no. 3, 613--632.
\bibitem{PSS2} Phong, D.H., Song, J., Sturm, J. {\em Degeneration of K\"ahler-Ricci solitons on Fano manifolds}, arXiv:1211.5849.
\bibitem{PSSW} Phong, D.H., Song, J., Sturm, J., Weinkove, B. {\em  The K\"ahler-Ricci flow and the $\db$ operator on vector fields}, J. Differential Geom. {\bf 81} (2009), no. 3, 631--647.
\bibitem{Ru} Rubinstein, Y.A. {\em On the construction of Nadel multiplier ideal sheaves and the limiting behavior of the Ricci flow},  Trans. Amer. Math. Soc. {\bf 361} (2009), no. 11, 5839--5850.
\bibitem{ST} \v Se\v sum, N., Tian, G. {\em Bounding scalar curvature and diameter along the K\"ahler Ricci flow (after Perelman)}, J. Inst. Math. Jussieu {\bf 7} (2008), no. 3, 575--587.
\bibitem{SZ} Shi, Y., Zhu, X. {\em  An example of a singular metric arising from the blow-up limit in the continuity approach to K\"ahler-Einstein metrics}, Pacific J. Math. {\bf 250} (2011), no. 1, 191--203.
\bibitem{Sz} Sz\'ekelyhidi, G. {\em Greatest lower bounds on the Ricci curvature of Fano manifolds}, Compos. Math. {\bf 147} (2011), no. 1, 319--331.
\bibitem{Ti} Tian, G. {\em On K\"ahler-Einstein metrics on certain K\"ahler manifolds with $C\sb 1(M)>0$},
Invent. Math. {\bf 89} (1987), no. 2, 225--246.
\bibitem{Ti3} Tian, G. {\em On Calabi's conjecture for complex surfaces with positive first Chern class}, Invent. Math. {\bf 101} (1990), no. 1, 101--172.
\bibitem{Ti2} Tian, G. {\em K\"ahler-Einstein metrics with positive scalar curvature}, Invent. Math. {\bf 130} (1997), no. 1, 1--37.
\bibitem{Ti4} Tian, G. {\em K-stability and K\"ahler-Einstein metrics}, arXiv:1211.4669.
\bibitem{TW} Tian, G., Wang, B. {\em On the structure of almost Einstein manifolds}, arXiv:1202.2912.
\bibitem{TZ} Tian, G., Zhu, X. {\em Convergence of K\"ahler-Ricci flow}, J. Amer. Math. Soc. {\bf 20} (2007), no. 3, 675--699.
\bibitem{To} Tosatti, V. {\em K\"ahler-Ricci flow on stable Fano manifolds}, J. Reine Angew. Math. {\bf 640} (2010), 67--84.
\bibitem{Ya} Yau, S.-T. \emph{On the Ricci curvature of a compact K\"ahler manifold and the complex Monge-Amp\`ere equation, I}, Comm. Pure Appl. Math. {\bf 31} (1978), 339--411.
\bibitem{Y2} Yau, S.-T. {\em Open problems in geometry}, Proc. Sympos. Pure Math. {\bf 54} (1993), 1-28 (problem 65).
\bibitem{Ye} Ye, R. {\em The logarithmic Sobolev inequality along the Ricci flow}, preprint, arXiv:math/0707.2424.
\bibitem{Zh} Zhang, Q.S. \emph{A uniform Sobolev inequality under Ricci flow}, Int. Math. Res. Not. IMRN {\bf 2007},  no. 17, Art. ID rnm056, 17 pp.
 and erratum in Int. Math. Res. Not. IMRN {\bf 2007},  no. 19, Art. ID rnm096, 4 pp.
\end{thebibliography}
\end{document}